\documentclass[10pt]{article}

\usepackage{amsmath}
\usepackage{amssymb}

\usepackage{cite}

\usepackage{amsthm}

\usepackage{lineno}

\usepackage{microtype}
\DisableLigatures[f]{encoding = *, family = * }

\usepackage[labelfont=bf,labelsep=period,justification=raggedright]{caption}

\makeatletter
\renewcommand{\@biblabel}[1]{\quad#1.}
\makeatother

\date{}

\newcommand{\be}{\begin{equation}}
\newcommand{\ee}{\end{equation}}
\newcommand{\bea}{\begin{eqnarray}}
\newcommand{\eea}{\end{eqnarray}}

\newcommand{\rme}{{\rm{e}}}

\newcommand{\iGAi}{{i}}
\newcommand{\iGAj}{{j}}

\newcommand{\reversion}[1] { {#1}^{\dagger} }
\newcommand{\cliffconj}[1] { \bar{#1} }

\newtheorem*{defn}{Definition}
\newtheorem*{thm}{Theorem}

\begin{document}

\begin{flushleft}
{\Large
\textbf{Functions of multivector variables}
}
\\
James M.~Chappell$^{1,\ast}$, 
Azhar Iqbal$^{2}$, 
Lachlan J.~Gunn$^{3}$,
Derek Abbott$^{4}$
\\
\bf{1} James M.~Chappell, School of Electrical and Electronic Engineering, University of Adelaide, SA
5005, Australia
\\
\bf{2} Azhar Iqbal, School of Electrical and Electronic Engineering, University of Adelaide, SA
5005, Australia
\\
\bf{3} Lachlan J.~Gunn, School of Electrical and Electronic Engineering, University of Adelaide, SA
5005, Australia
\\
\bf{4} Derek Abbott, School of Electrical and Electronic Engineering, University of Adelaide, SA
5005, Australia
\\
$\ast$ E-mail: james.chappell@adelaide.edu.au
\end{flushleft}

\section*{Abstract}
As is well known, the common elementary functions defined over the real numbers can be generalized to act not only over the complex number field but also over the skew (non-commuting) field of the quaternions.  In this paper, we detail a number of elementary functions extended to act over the skew field of Clifford multivectors, in both two and three dimensions. Complex numbers, quaternions and Cartesian vectors can be described by the various components within a Clifford multivector and from our results we are able to demonstrate new inter-relationships between these algebraic systems. One key relationship that we discover is that a complex number raised to a vector power produces a quaternion thus combining these systems within a single equation. We also find a single formula that produces the square root, amplitude and inverse of a multivector over one, two and three dimensions.  Finally, comparing the functions over different dimension we observe that $ C\ell \left (\Re^3 \right ) $ provides a particularly versatile algebraic framework.


\section*{Introduction}

Clifford algebras are associative non-commutative algebras developed by William~K.~Clifford around 1878 building on the exterior algebras developed earlier by Hermann~Grassmann.
Specifically, denoting $ \bigwedge \Re^{n}  $ as the exterior algebra of $ \Re^n $ then we produce the space of multivectors $ \Re \oplus \Re^n \oplus \cdots \oplus \bigwedge^n \Re^n $ denoted by $ C\ell \left (\Re^n \right ) $ with unity 1. These algebras can be either simple, hence isomorphic to matrix algebras over the reals, complex numbers, or quaternions, or semisimple, and hence isomorphic to the direct sum of two matrix algebras over the reals or quaternions~\cite{Lounesto2001,Porteous1995,Hahn1994,Lam1973,Hestenes111}.

In this paper, we firstly describe some general results applicable in $ C\ell \left (\Re^n \right ) $, before exploring the elementary functions based on multivectors in two and three dimensions, which then finally allows us to identify several unifying relationships.  Clifford multivectors form a generalization of the elementary functions over complex and quaternionic numbers~\cite{Pearson2005} that can be recovered as special cases.

\section*{Analysis}

\subsection*{General results for multivectors in $ C\ell \left (\Re^n \right ) $}

Within $ C\ell \left (\Re^n \right ) $ we form a multivector $ \Re \oplus \Re^n \oplus \cdots \oplus \bigwedge^n \Re^n $ that we can write as
\be \label{MnD}
M = A_0 + A_1 + A_{2} + A_{3}  + \cdots + A_n ,
\ee
where $ A_0 \in \Re $, $ A_1 \in \Re^n $,  $ A_2 \in \bigwedge^2 \Re^n , \dots , A_n \in \bigwedge^n \Re^n $.  The following definitions for the general case of multivectors over $ C\ell \left (\Re^n \right ) $ are essentially as found in \cite{Hestenes111}.

\begin{defn}[Grade selection]
We define the grade selection operation $ \langle M \rangle_k = A_k \in \bigwedge^k \Re^n $.  
The number of elements in each grade $ A_k $ follows the Pascal triangle relation $ \frac{n!}{k! (n-k)!} $ with the $ n + 1 $ grades forming a $ 2^n$-dimensional real vector space.
\end{defn}

\begin{defn}[Orthonormal basis]
A set of orthonormal basis elements $ e_k $ for $ C\ell \left (\Re^n \right ) $, with $ j,k \in \{1,2, \dots , n \} $ satisfy
\be
e_k e_k = 1, e_j e_k = - e_k e_j, j \ne k .
\ee
\end{defn}
For example, in  $ C\ell \left (\Re^3 \right ) $ we have the basis elements $ e_1, e_2 ,e_3 $ forming a multivector $ M = A_0 + A_1 + A_2 + A_3 $ with $ A_0 = a_0 $, $ A_1 = a_1 e_1 + a_2 e_2 + a_3 a_3 $, $ A_2 = a_4 e_1 e_2 + a_5 e_3 e_1 + a_6 e_2 e_3 $ and $ A_3 = a_7 e_1 e_2 e_3 $, where $ a_{0, \dots , 2^n-1} \in \Re $. In order to abbreviate notation we often write $ e_{12} \equiv e_1 e_2 $ and $  e_{123} \equiv e_1 e_2 e_3  $ etc.

\begin{defn}[Multivector involutions]
We define three involutions on a multivector $ M $: firstly {\bfseries space inversion} written as $ M^* $ defined by $ e_k \rightarrow - e_k $, secondly {\bfseries reversion} written as $ \reversion{M} $ that reverses the order of all products, $ e_1 e_2 \cdots e_n \rightarrow e_n e_{n-1} \cdots e_1 $ and thirdly a composition of the first two that forms {\bfseries Clifford conjugation}  written as $ \cliffconj{M} = M^{\dagger * } $. This produces a variation in signs over the different grades as follows
\bea
M^* & = & A_0 - A_1 + A_{2} - A_{3} + A_4 - A_5 + A_6 - A_7 + A_8 \cdots + (-)^n A_n \\ \nonumber
\reversion{M} & = & A_0 + A_1 - A_{2} - A_{3} + A_4 + A_5 - A_6 - A_7  + A_8 + \cdots + (-)^{\left\lfloor n/2 \right \rfloor} A_n \\ \nonumber
\cliffconj{M} & = & A_0 - A_1 - A_{2} + A_{3} + A_4 - A_5 - A_6 + A_7 + A_8 + \cdots + (-)^{n + \left\lfloor n/2 \right \rfloor} A_n\nonumber .
\eea
\end{defn}

Addition and subtraction of multivectors involves adding and subtracting the corresponding terms of the algebra and multiplication is through the formal application of the law of the distribution of multiplication over addition, that is explicated in the sections on two and three dimensional multivectors to follow.
We find that {\it reversion } and Clifford {\it conjugation} are anti-automorphisms producing $ \reversion{(M_1 M_2)} = \reversion{M_2} \reversion{M_1} $ and $ \overline{M_1 M_2} = \cliffconj{M_2} \cliffconj{M_1} $ whereas {\it space inversion} $ \left ( M_1 M_2 \right )^* = M_1^* M_2^* $ is an automorphism.

Note, that using the reversion involution, calculating the corresponding grades in $ M \reversion{M} $ we find that all products are of the form $ e_1 e_2 \cdots e_n e_n e_{n-1} \cdots e_1 = +1 $. Hence we can use the reversion involution to form a positive definite scalar $ \left \langle M \reversion{M} \right \rangle_0 $.  This leads us to define an inner product for multivectors.

\begin{defn}[Inner product]
We define for two multivectors $ M_1 $ and $ M_2 $ the product
\be \label{innerProductN}
\left \langle M_1 \reversion{M_2} \right \rangle_0 = \left \langle M_2 \reversion{M_1} \right \rangle_0 ,
\ee
which can be shown to have the required properties for an inner product. 
This induces a norm on a multivector 
\be
||M||^2 = \left \langle M \reversion{M} \right \rangle_0 ,
\ee
which is positive definite as required.  Conventional results now follow, such as a triangle inequality for multivectors.
\end{defn}

\begin{defn}[Square root]
The square root of a multivector $ Y $ is the multivector $ M $ such that $ Y = M^2 $ and we write $ M = Y^{1/2} $.  
\end{defn}
We reserve the square root symbol $ \sqrt{} $ to act over the reals and complex-like numbers, with its conventional definition, producing a value within the complex-like numbers.

\begin{defn}[Multivector amplitude]
We define the {\it{amplitude}} of a multivector $ M $ as
\be \label{magnitude2D}
|M| = \sqrt{M \cliffconj{M}} .
\ee
Note that $ M \cliffconj{M} $ is not positive definite and does not have a value in the reals in general and hence the amplitude may not exist in all cases.
\end{defn}

\begin{defn}[Multivector exponential]
The exponential of a multivector is defined by constructing the Taylor series 
\be \label{exponentialMultivector}
\rme^M = 1 + M +\frac{M^2}{2!} + \frac{M^3}{3!} + \dots,
\ee
which is absolutely convergent for all multivectors \cite{Hestenes111}.  
\end{defn}
Convergence is easily demonstrated because $ || \langle M^{n} \rangle_k || < || M^n || < || M ||^{n} $, so that if the normed series converges then each grade of the series must converge. The infinite sequence $ \{ M_n \} $ of multivectors $ M_1, M_2, M_3, \dots , M_n, \dots $ approaches the multivector $ L $ as a limit, that is $ M_n \rightarrow L $, if $ ||L - M_n || \rightarrow 0 $ as $ n \rightarrow \infty $.

\begin{defn}[Logarithm]
The logarithm of a multivector is defined as the inverse of the exponential function.  For a given multivector $ Y $ we find $ M $, such that $ Y = \rme^M $ and we write $ M = \log Y $, which is multivalued in general. Hence we have $ \rme^{ \log Y } = Y $. The principal value of the logarithm can be defined as the multivector $ M = \log Y  $ with the smallest norm.  
\end{defn}

In even dimensional spaces $ C\ell \left (\Re^{2 n} \right ) $ the pseudoscalar is non-commuting whereas in odd dimension the pseudoscalar is commuting.  Additionally spaces of dimension $ 2, 3, 6, 7, 10, 11, \cdots $ have a pseudoscalar that squares to minus one whereas $ 4, 5, 8, 9, 12, 13, \cdots $ the pseudoscalar squares to plus one.  Hence spaces that have a commuting pseudoscalar that squares to minus one lie in spaces of dimension $ 3, 7, 11, 15, \cdots, 4 n - 1, \cdots $, where $ n \in \mathbb{N} $.  In general these pseudoscalar properties have period four.

\begin{defn}[Hyperbolic trigonometric functions]
Splitting the exponential series,as shown in Eq.~(\ref{exponentialMultivector}), into odd and even terms we define the hyperbolic trigonometric functions
\bea \label{HyperTrigN}
\cosh M & = & \sum_{n=0}^{\infty}  \frac{M^{2n}}{(2n)!} = \frac{1}{2} \left ( \rme^{ M} +  \rme^{- M} \right )  , \\ \nonumber
\sinh M & = &  \sum_{n=0}^{\infty}  \frac{M^{2n+1}}{(2n+1)!} = \frac{1}{2 } \left ( \rme^{ M} -  \rme^{- M} \right )  .  \nonumber
\eea
The exponential form immediately implies $ \rme^M = \cosh M + \sinh M $ and we can then easily confirm the usual results that $ \sinh 2 M = 2 \sinh M \cosh M $ and $ \cosh^2 M - \sinh^2 M = 1 $.
\end{defn}

\begin{defn}[Trigonometric functions]
We define the trigonometric functions with the alternating series
\be \label{TrigN}
\cos M =  \sum_{n=0}^{\infty}  \frac{(-)^{n} M^{2n}}{(2n)!} , \, \sin M = \sum_{n=0}^{\infty}  \frac{(-)^{n} M^{2n+1}}{(2n+1)!}  . 
\ee
This definition then implies $ \cos^2 M + \sin^2 M = 1 $.
\end{defn}
We can write the trigonometric functions in an exponential form, such as $ \cos M = \frac{1}{2} \left ( \rme^{ J M} +  \rme^{- J M} \right ) $ for example, provided we have a commuting pseudoscalar with $ J^2 = -1 $. This is only true though in spaces of dimension $ 3,7,11,\dots $, as previously discussed.

For the multivector finite series $ S_n = 1 + M + M^2 + \dots + M^n $ we find $ M S_n = M + M^2 + \dots + M^{n+1} $ and so  $ S_n - M S_n = (1 - M ) S_n = 1 - M^{n+1} $.  Multiplying on the left with the inverse of $ (1 - M ) $ we find for the sum
\be
S_n = (1-M)^{-1} \left ( 1 - M^{n+1} \right ) ,
\ee
provided the inverse exists.

\subsection*{Clifford's geometric algebra of two dimensions}

Within Clifford's geometric algebra $ C\ell \left (\Re^2 \right ) $, we form a multivector $ M \in \Re \oplus \Re^2 \oplus \bigwedge^2 \Re^2 $ that can be expressed in terms of an orthonormal basis as
\be \label{M2D}
M = a + x e_1 + y e_2 + b e_{12} ,
\ee
where $ a, x, y, b $ are real scalars and the bivector defined as $ e_{12} = e_1 e_2 $.  

The space of multivectors in $ C\ell \left (\Re^2 \right ) $ is isomorphic to the matrix algebra $ C\ell \left (\Re^2 \right ) \cong {\rm{Mat}}(2,\Re) $.  We have the bivector $ e_{12}^2 = e_1 e_2 e_1 e_2 = - e_1 e_1 e_2 e_2 = -1 $. Hence the even subalgebra\footnote{The subalgebra of $ C\ell_{2} $ spanned by 1 and $ e_{12} $, consisting of scalar and bivector components forming the even subalgebra, with $ e_{12} $ taking the role of the unit imaginary, is isomorphic to $ {\mathbb{C}} $. } in two dimensions, given by $ a + b e_{12} $, is isomorphic to the complex field, and so we can assume the results from complex number theory when the multivector lies within this restricted domain. For example, the $ \log $ of a multivector in the even subalgebra $ \log (a + e_{12} b ) = \log \sqrt{a^2+b^2} + \theta e_{12} $, with the multivalued $ \theta = \arctan( b/a ) $, as found in complex number theory. 
In addition to the even subalgebra representing the complex numbers, we also have the subalgebra $ a + x e_1 $ forming the one-dimensional Clifford algebra $ C\ell \left (\Re^1 \right ) $. 

The sum or difference of two multivector numbers $ M_1 = a_1 + x_1 e_1 + y_1 e_2 + b_1 e_{12} $ and $ M_2 = a_2 + x_2 e_1 + y_2 e_2 + b_2 e_{12} $ is defined by
\be
M_3 = M_1 \pm M_2 = a_1 \pm a_2 +(x_1 \pm x_2 ) e_1 +(y_1 \pm y_2 ) e_2 + (b_1 \pm b_2 ) e_{12} .
\ee
The product $ M_3 $ of multivectors $ M_1 $ and $ M_2 $ is found through the formal application of the distributive law of multiplication over addition
\bea
M_3 = M_1 M_2 & = & a_1 a_2 + x_1 x_2 + y_1 y_2 - b_1 b_2 + (a_1 x_2 + a_2 x_1 + b_1 y_2 - y_1 b_2 ) e_1 \\ \nonumber
& & + (a_1 y_2 + y_1 a_2 + x_1 b_2 - b_1 x_2 ) e_2 + (a_1 b_2 + b_1 a_2 + x_1 y_2 - y_1 x_2 ) e_{12} . \nonumber
\eea

In two dimensions the {\it conjugation} involution produces
\be \label{dagger2D}
\cliffconj{M} = a - x e_1 - y e_2 - b e_{12} .
\ee
In terms of multiplication and additions we can write $ \cliffconj{M} = - \frac{1}{2} \left ( M - e_1 M e_1 - e_2 M e_2 + e_{12} M e_{12}  \right ) $.  We then have the scalar part of a multivector $ \langle M \rangle_0 = \frac{1}{2} \left ( M + \cliffconj{M} \right ) $ and the sum of vector and bivector components $ \langle M \rangle_1 + \langle M \rangle_2 = \frac{1}{2} \left ( M - \cliffconj{M} \right ) $.  If required, we can also isolate the vector components of M as $ \langle M \rangle_1 = \frac{1}{2} \left ( M + e_{12} M e_{12} \right ) = v_1 e_1 + v_2 e_2  $.
Using Clifford conjugation we then find
\bea \label{magsqrd}
M \cliffconj{M} & = & \cliffconj{M}  M \\ \nonumber
 & = & (a + x e_1 + y e_2 + b e_{12} )(a - x e_1 - y e_2 - b e_{12} ) \\ \nonumber
& = & a^2 - x^2 - y^2 + b^2 , \nonumber 
\eea
producing a real number, though not necessarily non-negative.

\begin{defn}[Negative square root]
We define the principal square root of negative numbers in two dimensions as follows: given a real number $ a \in \Re \ge 0 $ we define
\be \label{principalRoot2D}
\sqrt{-a} = e_{12} \sqrt{a} ,
\ee
using the property that the bivector squares to minus one.
\end{defn}

The {\it{amplitude}} of a multivector in two dimensions becomes
\be \label{magnitude2D}
|M| = \sqrt{M \cliffconj{M}} = \sqrt{a^2 - x^2 - y^2 + b^2} .
\ee
Note that that the special case of $ x = y = 0 $ produces the magnitude of a complex-like number.

The {\it reversion } involution in two dimensions produces
\be \label{reversion2D}
\reversion{M} =  a + x e_1 + y e_2 - b e_{12} ,
\ee
which we can also write algebraically as $ \reversion{M} =  \frac{1}{2} \left ( M + e_1 M e_1 + e_2 M e_2 + e_{12} M e_{12}  \right ) $.
We then find the norm of a multivector in two dimensions
\be
|| M || = \left \langle M \reversion{M} \right \rangle_0^{1/2} = \left \langle a^2 + x^2 + y^2 + b^2 + 2 a (x e_1 + y e_2 ) \right \rangle_0^{1/2} =\sqrt {a^2 + x^2 + y^2 + b^2 }.
\ee

Also, this definition of the product and the definition of the amplitude in Eq.~(\ref{magsqrd}) produces the homomorphism
\be \label{norm2D}
| M_1 M_2 |^2 = M_1 M_2 \overline{M_1 M_2} = M_1 M_2 \cliffconj{M_2} \cliffconj{M_1} = |M_1|^2 |M_2|^2 .
\ee
Expanding this expression in full we have
\bea
 & & (a_1^2 - a_2^2 - a_3^2 + a_4^2) (b_1^2 - b_2^2 - b_3^2 + b_4^2) = (a_1 b_1 + a_2 b_2 + a_3 b_3 - a_4 b_4)^2  \\ \nonumber
& &  - (a_1 b_2 + a_2 b_1 - a_3 b_4 + a_4 b_3 )^2 - (a_1 b_3 + a_2 b_4 + a_3 b_1 - a_4 b_2 )^2 + (a_1 b_4 + a_2 b_3 - a_3 b_2 + a_4 b_1   )^2 \nonumber
\eea
and so is a variation of Euler's four-square identity.
It should be noted that  $ |M_1|^2 |M_2|^2 = \pm \left (|M_1| |M_2| \right )^2 $ and taking roots we find that $ |M_1 M_2 | = \pm | M_1 | | M_2 | $.  This is analogous to the result from complex number theory that $ (-1)^{1/2} \times (-1)^{1/2} = \pm (-1 \times -1 )^{1/2} $.

Also, from Eq.~(\ref{magsqrd}) we can see that because $ |M|^2 = M \cliffconj{M} $ is a real scalar, we can define the inverse multivector as
\be \label{inverse2D}
M^{-1} = \cliffconj{M}/|M|^2 .
\ee
This gives $ M M^{-1} = M \cliffconj{M}/|M|^2 = |M|^2/|M|^2 = 1 $ as required.  This now allows us to define the division operation $ M_1/M_2 = M_1 M_2^{-1} $.  Clearly, a multivector fails to have an inverse if $ M \cliffconj{M} = a^2 - x^2 - y^2 + b^2 = 0 $ and so fails to form a division algebra in these cases.  This expression for the inverse is analogous to the formula for the inverse of a complex number $ z^{-1} = \cliffconj{z}/|z|^2 $, that can be recovered as a special case from Eq.~(\ref{inverse2D}) for $ M $ in the even subalgebra.

Now, for more complex manipulations to follow it is preferable to write the general multivector as
\be \label{generalMultivector}
M = a + \mathbf{v} + \iGAi b ,
\ee
where $ \mathbf{v} = x e_1 + y e_2 $ defines a vector, with the bivector $ \iGAi = e_1 e_2 $. We also define $ F = \mathbf{v} + \iGAi b $ so that we can write $ M = a + F $. We have used the symbol $ \iGAi $ for the pseudoscalar\footnote{The pseudoscalar refers to the highest dimensional element of the algebra, which is of dimension $ n $ for a Clifford algebra $ C\ell \left (\Re^n \right ) $.} that is also commonly used for the unit imaginary $ \sqrt{-1} $. This notation is adopted because complex numbers also lie in a two-dimensional space analogous to the even subalgebra of the two-dimensional multivector. We have the important result that $ \mathbf{v}^2 = (x e_1 + y e_2)(x e_1 + y e_2) = x^2 + y^2 $and so a real scalar giving the Pythagorean length.
Hence, using this notation, the condition for a multivector inverse to exist is given by $ a^2 + b^2 \ne \mathbf{v}^2 $.

\subsubsection*{The square root}

The square roots of a multivector in $ C\ell \left(\Re^2 \right ) $ are given by
\be \label{squareRoot2D}
M^{\frac{1}{2}} = \frac{1}{\sqrt{2 (a \pm |M|)}} \left ( M \pm |M| \right ) = \frac{1}{  \sqrt{M + \cliffconj{M} \pm 2 |M|}} ( M \pm |M| ) ,
\ee
provided $ M \notin \Re $.
Proof: Given a multivector $ S = c + \mathbf{w} + \iGAi d $ we find $ S^2 = c^2 + \mathbf{w}^2 - d^2 + 2 c (\mathbf{w} + \iGAi d) $. Hence, provided $ c \ne 0 $ implying vector or bivector components are present, the root of a multivector $ M = a + \mathbf{v} + \iGAi b $ must be of the form $ S = c + \frac{1}{2 c} (\mathbf{v} + \iGAi b ) = \frac{1}{2 c} (2 c^2 + \mathbf{v} + \iGAi b ) = \frac{1}{2 c} (2 c^2 - a + M )  $. It just remains now to find $ c $.  The scalar component of the equation $ S^2 = M $  gives us $ c^2 = \frac{1}{2} (a \pm |M| ) $.  Substituting this expression we find $ S = \frac{1}{\sqrt{2(a \pm |M|)}} (M \pm |M|) $ as required.  However if $ M $ is a pure real number then we need to consider the special case with $ c = 0 $.  This gives $ S^2 = \mathbf{w}^2 - d^2 $ and so provides the roots of scalars.  If we are seeking the square root of a negative real $ -a $, where $ a \ge 0 $, then solving for $ d $  we find the root
\be \label{generalRootScalar}
\sqrt{-a} = \mathbf{v} \pm e_{12} \sqrt{a + \mathbf{v}^2 } ,
\ee
which is satisfied for all vectors $ \mathbf{v} = v_1 e_1 + v_2 e_2 $. The special case with $ \mathbf{v} = 0 $ and $ r = 1 $ produces the principal root defined earlier in Eq.~(\ref{principalRoot2D}). The possible roots of minus one in Clifford multivectors has been further investigated elsewhere as in \cite{Sangwine2006}. For the roots of positive reals we solve instead for the vector length giving $ a^{\frac{1}{2}} = \pm \sqrt{a + d^2} \hat{\mathbf{v}} + e_{12} d  $.  This last expression shows the need to distinguish the square root of reals  given by $ \sqrt{a} $ and the more general square roots over the domain of multivectors shown as $ a^{\frac{1}{2}} $ in order to avoid circular definitions.  The principal values though will correspond with each other.

The square root of a multivector will not exist unless $ | M | $ is real, that is $ \mathbf{v}^2 \leq a^2 + b^2 $, because the pseudoscalar $ \iGAi $ is non-commuting.
For example the square root fails to exist for a pure vector with $ M = \mathbf{v} $ giving $ |M| = \sqrt{-\mathbf{v}^2 } = \iGAi \sqrt{\mathbf{v}^2 } $.  
 We can find from Eq.~(\ref{squareRoot2D}) two roots, each of which though can also be negative, so therefore in general Eq.~(\ref{squareRoot2D}) produces four possible square roots.  The last version on the right has the advantage of being expressed in $ M $ alone and not in components.  For the root with the negative sign to exist we require two extra conditions: $ a > 0 $ and  $ a - |M| > 0 $ or $ \mathbf{v}^2 > b^2 $. 

For the special case of complex-like numbers with a multivector $ z = a + \iGAi b $ we have
\be \label{squareRoot2DComplex}
z^{\frac{1}{2}} = \frac{z \pm |z|}{\sqrt{z + \cliffconj{z} \pm 2 |z|}},
\ee
which agrees with results from complex number theory.

\subsubsection*{Trigonometric form of a multivector}

\begin{defn}[Multivector argument]
We define the argument of a multivector $ M = a + \mathbf{v} + \iGAi b =a + F $ as  
\be
\arg M = {\rm{arctan}} \left ( \frac{|F|}{a} \right ) ,
\ee
where $ | F | = \sqrt{b^2 -\mathbf{v}^2 } $.
The function is multivalued modulo $ 2 \pi $ and also depends on which quadrant the point $ (a,|F|) $ is in. We define the principal value of the argument $ -\pi < \phi \le \pi $.
\end{defn}

\begin{thm}[Trigonometric form]
A two dimensional multivector can be written in the form
\be \label{polarForm2DAssert}
M = a + \mathbf{v} + \iGAi b = \left ( \cos \phi +  \frac{\mathbf{v} + \iGAi b }{\sqrt{b^2 - \mathbf{v}^2 }}  \sin \phi \right ) |M|  = \left ( \cos \phi +  \hat{F}  \sin \phi \right ) |M| ,
\ee
where $ \phi = \arg M $, and we have defined $ \hat{F} = F/|F| $, provided $ |M|, |F| \ne 0 $ . 
\end{thm}
\begin{proof}
Assuming $ b^2 > \mathbf{v}^2 $ and $ |M| \ne 0 $, we have $ \cos \phi = a/|M| $ and $ \sin \phi = \sqrt{b^2 - \mathbf{v}^2}/|M| $. Substituting we find $ M = a + \mathbf{v} + \iGAi b $ as required.  Alternatively if $ b^2 < \mathbf{v}^2 $ then $ \sqrt{b^2 -\mathbf{v}^2 } $ becomes a bivector but because it will cancel with the same term in $ \sin \phi $ the multivector will be returned.  Likewise if $ | M| $ is a bivector, then this will also cancel with $ | M| $ in the $ \sin \phi $ and $ \cos \phi $ terms. Hence Eq.~(\ref{polarForm2DAssert}) applies provided $ |M| \ne 0 $ and $ b^2 \ne \mathbf{v}^2 $, as required.
\end{proof}

The order of the factors in Eq.~(\ref{polarForm2DAssert}) is important because $ \phi $ and $ |M| $ can lie in the even subalgebra and so will not necessarily commute with $  \mathbf{v} + \iGAi b $, in general. Notationally, it is also important to note that we define $ \hat{F} =  \frac{\mathbf{v} + \iGAi b }{\sqrt{b^2 - \mathbf{v}^2 }} \equiv (  \mathbf{v} + \iGAi b )/ \sqrt{b^2 - \mathbf{v}^2} $, where the denominator always follows the numerator, due to commutativity issues. Note, it turns out that we can rearrange the factors to produce an equivalent form $ M = |M| \left ( \cos \phi + \sin \phi (b^2 - \mathbf{v}^2)^{-\frac{1}{2}} (  \mathbf{v} + \iGAi b )  \right ) $. 

Now, assuming the trigonometric form in Eq.~(\ref{polarForm2DAssert}) exists, we find for integer powers $ p $ that
\be \label{Mp2D}
M^p =   \left ( \cos \phi + \hat{F}  \sin \phi \right )^p |M|^p = \left ( \cos p \phi +\hat{F} \sin p \phi \right ) |M|^p ,
\ee
a generalization of de~Moivre's theorem for multivectors, valid for $ | F|, |M | \ne 0 $.

Now, because multivector multiplication is associative we can find the rational  powers $ M^{\frac{p}{2^q}} $, where $ p,q $ are integers.  We will now see how this relation can be written in polar form using the exponential map, which will allow us to calculate more general multivector powers using logarithms.

\subsubsection*{Exponential map of a multivector}

Given a two-dimensional multivector $ a + \mathbf{v} + \iGAi b = a + F $, we find $ F^2 = (\mathbf{v} + \iGAi b)^2 = \mathbf{v}^2 - b^2  $ and so $  | F |^2= F \cliffconj{F} = -F^2  $.  Hence, given the exponential map in Eq.~(\ref{exponentialMultivector}), we find
\bea \label{derivationExpMultivector2D}
\rme^{a + \mathbf{v} + \iGAi b} & = & \rme^{a } \rme^{\mathbf{v} + \iGAi b} = \rme^{a } \rme^{ F }  \\ \nonumber
& = & \rme^{a } \left ( 1 + F - \frac{| F |^2 }{2!}- \frac{ F | F |^2}{3!} + \frac{| F |^4 }{4!} + \dots  \right ) . \nonumber
\eea
If $ | F | = 0 $, then referring to the last line of the derivation above, we see that all terms following $  F $ will be zero, and so, in this case $ \rme^{a + \mathbf{v} + \iGAi b} = \rme^{a } (1+ \mathbf{v} + \iGAi b) $. 
Now, assuming the power series definitions for the trigonometric functions we can then find the closed form
\be
\rme^{a + \mathbf{v} + \iGAi b} = \rme^{a } \left ( \cos | F | + \hat{F}  \sin | F |   \right ) , 
\ee
a result that remains valid even if $ | F | = \sqrt{b^2 - \mathbf{v}^2 } $ is a bivector, because as we know from complex number theory the trigonometric functions will simply become hyperbolic trigonometric functions.

We can thus rearrange this result, to write a multivector in polar form as
\be \label{polarForm2D}
M = a + \mathbf{v} + \iGAi b  =  | M | \rme^{ \hat{F} \phi }  =  | M | \rme^{ (\mathbf{v} + \iGAi b ) /|F| \phi }  ,
\ee
where $ \phi = \arg M $.  We find that an exponential form is only possible if $ |M| $ is real, even though the trigonometric form, shown previously in Eq.~(\ref{polarForm2DAssert}), is valid generally.  This is because $ | \rme^{a + \mathbf{v} + \iGAi b} | = \left (\rme^{a + \mathbf{v} + \iGAi b} \rme^{a - \mathbf{v} - \iGAi b} \right )^{1/2} = e^{a} $, a result that is always real, whereas in general $ | a + \mathbf{v} + \iGAi b | $ can become a bivector. This also explains why the square root fails to exist in these cases. Eq.~(\ref{polarForm2D}) is a generalization of the exponential form for complex numbers. That is, for $ \mathbf{v} = 0 $, we have $ M = a + \iGAi b  =  | M | \rme^{ \hat{F} \phi}  =  | M | \rme^{  \iGAi \phi} = \sqrt{a^2 + b^2 } \left ( \cos \phi + \iGAi \sin \phi \right ) $, where $ \iGAi $ is the bivector, but equivalent to the logarithm for complex numbers. Hence the logarithm of a multivector $ M $ becomes
\be \label{LogMulti2D}
\log M = \log  |M| +  \hat{F } \phi ,
\ee
$ \phi = \arg M $. The logarithm multivaluedness coming from the argument function. 

We can now also define the multivector power
$ M^P = \rme^{ \log ( M ) P } $, where $ P $ is a also general multivector and, due to non-commutativity, alternatively as $ \rme^{ P \log ( M ) } $.

\subsubsection*{Trigonometric functions of a multivector}

In two dimensions, the expressions for the hyperbolic trigonometric functions given in Eq.~(\ref{HyperTrigN}) can be simplified to give
\bea \label{hyperTrig2D}
\cosh M & = &  \frac{1}{2} \left ( \rme^{ a + F} +  \rme^{- a -F} \right ) = \cos |F| \cosh a + \hat{F} \sin |F| \sinh a \\ \nonumber
\sinh M & = & \frac{1}{2} \left ( \rme^{ a + F} -  \rme^{- a -F} \right ) = \cos |F| \sinh a + \hat{F} \sin |F| \cosh a  .\nonumber
\eea
We can view these relations as a generalization of the results for complex numbers.  For example, for complex numbers we have $ \cosh(a+ \iGAi b) = \cos b \cosh a  + \iGAi \sin b \sinh a  $, whereas for the case of multivectors we can write $ \cosh ( a + \mathbf{v} + \iGAi b ) = \cosh ( a + F) = \cosh ( a + \hat{F} |F|) $, and so produce the results of Eq.~(\ref{hyperTrig2D}), where $ \hat{F} $ now takes the role of the unit imaginary, because $ \hat{F}^2 = -1 $.  These results also remaining valid if $ |F| $ is a bivector.

Now, because the pseudoscalar $ \iGAi $ in two dimensions is not commuting there is no way to generate the alternating series shown in Eq.~(\ref{TrigN}) for the trigonometric functions from the exponential series using the pseudoscalar and so these will be developed in the next section in three dimensions.

Our complete list of results for multivectors in $ C\ell\left(\Re^2\right) $ are tabulated in Table~\ref{twoDFunctions}. The inverse hyperbolic trigonometric functions are also shown in Table~\ref{twoDFunctions}, using the algebraic procedure shown next in three dimensions.
In conclusion, we have identified several limitations in two dimensions, such as the lack of a commuting pseudoscalar, the nonexistence of the square root and exponential representation in a significant class of multivectors, however, we now produce the corresponding expressions with multivectors in the more general three-dimensional space where these limitations are absent.

\section*{The multivector in three dimensions}

In three dimensions we have the three basis elements $ e_1, e_2 $ and $ e_3 $, the three bivectors $ e_1 e_2 $, $ e_3 e_1 $ and $ e_2 e_3 $, as well as the trivector $ \iGAj = e_1 e_2 e_3 = e_{123} $ and we form the three dimensional geometric algebra $ C\ell \left (\Re^3 \right ) $.  In order to assist the readers intuition we note an isomorphism with matrix algebra that $ C\ell \left (\Re_3 \right ) \cong {\rm{Mat}}(2,C) $.  This isomorphism also implies that Clifford algebra shares the non-commuting and associativity properties of matrix algebra.  However it should be noted that the Clifford algebra we have defined over $ \Re^3 $ has more structure than is the case with the matrix definition, for example, we have a graded structure in $ C\ell \left (\Re_3 \right ) $ of scalars, vectors, bivectors and trivectors.  In three dimensions the trivector squares to minus one and commutes with all quantities and so in close analogy to the unit imaginary.  
Indeed, using the trivector we can also form what are called the dual relations, $ e_1 e_2 = \iGAj e_3 $, $ e_3 e_1 = \iGAj e_2 $ and $ e_2 e_3 = \iGAj e_1 $. Hence, we can write a general multivector in three dimensions as
\be \label{generalM3D}
M = a + \mathbf{v} + \iGAj \mathbf{w} + \iGAj t, 
\ee
where $ \mathbf{v} = v_1 e_1 + v_2 e_2 + v_3 e_3 $ and $ \mathbf{w} = w_1 e_1 + w_2 e_2 + w_3 e_3 $, which thus produces a multivector of eight dimensions.  The Clifford algebra $ C\ell \left ( \Re^3 \right ) $ contains the element $ j = e_{123} $ as a pseudoscalar such that the two dimensional subalgebra generated by $ \iGAj $ is the center $ Z ({\textsl{A}} ) $ of the algebra $ A = C\ell \left ( \Re^3 \right ) $. That is, every element of $ A $ commutes with every element of the centre $ Z (A ) $ that can be represented as $ a + \iGAj t $. Thus $ A $ is isomorphic to an algebra over the complex field. This is in contrast to $ C\ell(\Re^2) $ where the imaginary element $ \iGAi = e_{12} $ is not commuting with other elements of the algebra and so does not belong to the center $ {\rm{Cen}}(C\ell(\Re^2)) $.

Before proceeding to a general multivector product it is instructive to firstly calculate the special case of the product of two vectors $ \mathbf{v} $ and $ \mathbf{w} $.
Assuming the distribution of multiplication over addition we find
\bea \label{geometricProductRaw}
\mathbf{v} \mathbf{w} & = & (v_1 e_1 + v_2 e_2 + v_3 e_3)(w_1 e_1 + w_2 e_2 + w_3 e_3) \\ \nonumber
& = & v_1 w_1 + v_2 w_2 + v_3 w_3 + (v_2 w_3 - w_2 v_3 ) e_2 e_3 + (v_1 w_3 - v_3 w_1 ) e_1 e_3 + (v_1 w_2 - w_1 v_2 ) e_1 e_2 \\ \nonumber
& = & \textbf{v} \cdot \textbf{w}  + \textbf{v} \wedge \textbf{w}, \nonumber
\eea
consisting of the sum of the dot and wedge products, being a scalar and a bivector respectively.
In three dimensions we in fact have the relation $ \textbf{v} \wedge \textbf{w} = \iGAj \textbf{v} \times \textbf{w} $, where $ \iGAj $ is the trivector and $ \times $ is the vector cross product. For a vector squared, that is $ \textbf{v}^2 = \textbf{v} \textbf{v} $, we have $  \textbf{v} \wedge \textbf{v} = 0 $ and so $ \mathbf{v} \mathbf{v}  = \mathbf{v}  \cdot \mathbf{v}  = v_1^2 + v_2^2 + v_3^2 $ producing a scalar equal to the Pythagorean length squared.

Now, defining $ Z =  a + \iGAj t $ and $ F = \mathbf{v} + \iGAj \mathbf{w} $, we can write $ M = Z + F $, which splits the multivector into a component $ Z $ isomorphic to the complex number field and a multivector $ F $. 

For the multivector $ M $, we then have {\it Clifford conjugation }
\be
\cliffconj{M} = a - \mathbf{v} - \iGAj \mathbf{w} + \iGAj t = Z - F,
\ee
that produces the {\it{amplitude}} of a multivector in three dimensions
\be \label{magnitude3D}
| M | = \sqrt{M \cliffconj{M}} = \sqrt{a^2 - \mathbf{v}^2 + \mathbf{w}^2 - t^2 + 2 \iGAj (a t - \mathbf{v} \cdot \mathbf{w} ) },
\ee
that in general is a complex-like number.  We note that is well behaved with $ |M_1 M_2 |^2 = | M_1 |^2 | M_2 |^2 $ and $ |M_1 M_2 | = \pm | M_1 | | M_2 | $.  

\begin{defn}[Negative square root]
We define the principal square root when acting act over negative reals in $ C\ell \left (\Re^3 \right ) $ as follows: given a real number $ a \in \Re $ we define
\be \label{defnRoot3D}
\sqrt{-a} = e_{123} \sqrt{a} = \iGAj \sqrt{a} .
\ee
\end{defn}
In three dimensions the pseudoscalar $ \iGAj $ is commuting and so closely analogous to the scalar unit imaginary $ \sqrt{-1} $.
The subalgebra, consisting of quantities of the form $ a + \iGAj b $ form an isomorphism with the commuting complex numbers and we can therefore assume the results from complex number theory when restricted to this domain. However, if we allow the complex roots to expand to the domain of a multivector as given in Eq.~(\ref{generalM3D}), we need to solve 
\be
M^2 = a^2 + \mathbf{v}^2 - \mathbf{w}^2 - t^2 + 2 (a \mathbf{v} - t \mathbf{w})+ 2 \iGAj (t \mathbf{v} + a \mathbf{w} ) + 2 \iGAj ( a t + \mathbf{v} \cdot \mathbf{w} ) = c + \iGAj d ,
\ee
where $ c,d \in \Re $.
Solving this equation we find two distinct cases, either $ \mathbf{v} = \mathbf{w} = 0 $ that corresponds to the conventional square root over the complex numbers and $ a = t = 0 $ that provides a different set of roots over the domain of vectors and bivectors. That is, we find $ \left ( \mathbf{v}+ \iGAj \mathbf{w} \right )^2 = c + \iGAj d $, where $ c = \mathbf{v}^2 - \mathbf{w}^2 $ and $ d = 2 \mathbf{v} \cdot \mathbf{w} $. Hence we have an alternative set of roots for complex numbers as
\be
\left (c + \iGAj d \right )^{\frac{1}{2}} =  \mathbf{v}+ \iGAj \mathbf{w} .
\ee
 As a special case we can find for $ c = -1 $ and $ d = 0 $ 
\be
(-1)^{\frac{1}{2}} =  \sinh \theta \hat{w}^{\perp} + \iGAj \cosh \theta \hat{w} , 
\ee
where $ \hat{w} $ is a unit vector and $ \hat{w}^{\perp} $ is a unit vector perpendicular to $ \hat{w} $ and $ \theta \in \Re $. This equation also provides an alternative root of minus one to the trivector in Eq.~(\ref{defnRoot3D}). The investigation of roots within Clifford multivectors has been previously studied\cite{Sangwine2006}, and roots are simpler to analyze using the polar form of a multivector, investigated shortly.

In three dimensions we have the {\it{reversion}} involution
\be
\reversion{M} = a + \mathbf{v} - \iGAj \mathbf{w} - \iGAj t ,
\ee
giving
\be
M \reversion{M} = a^2 + \mathbf{v}^2 + \mathbf{w}^2 + t^2 + 2 (a \mathbf{v} - \iGAj \mathbf{v} \wedge \mathbf{w} + t \mathbf{w} ),
\ee
with the norm $ || M || = \langle M \reversion{M} \rangle_0^{1/2} = \sqrt{ a^2 + \mathbf{v}^2 + \mathbf{w}^2 + t^2} $.
Also when representing complex numbers in three dimensions using $ z = a + \iGAj b $ then the norm produces $ \sqrt{z \reversion{z}} = \sqrt{a^2 + b^2} $.

Now, because  $ M \cliffconj{M} $ is a commuting complex-like number, we can find the inverse multivector to $ M $ as
\be
M^{-1} = \cliffconj{M} /(M \cliffconj{M}),
\ee
which is the same definition as in the two-dimensional case.
The multivector inverse now fails to exist when $ M \cliffconj{M} = 0 $ or when $ a^2  + w^2 =  v^2 + t^2 $ and $ a t = \mathbf{v} \cdot \mathbf{w} $, which we can write as the single condition $ (\textbf{v} + \iGAj \textbf{w})^2 = (a + \iGAj t )^2 $ or $ F^2 = Z^2 $.
The inverse of a vector is a special case of this general multivector inverse, $ \mathbf{v}^{-1} = \mathbf{v}/\mathbf{v}^2 $.  The inverse obeys the relations $ (M^{-1})^{-1} = M $ and $ (M N)^{-1} = N^{-1} M^{-1} $.

Hamilton's quaternions $ {\rm{i}} ,{\rm{j}}, {\rm{k}} $, satisfying $ {\rm{i}}^2 = {\rm{j}}^2 = {\rm{k}}^2 = {\rm{i}} {\rm{j}} {\rm{k}} = -1 $, can be shown to be isomorphic to the even subalgebra of $ C\ell \left (\Re^3 \right ) $, so that a quaternion $ q = a + w_1 {\rm{i}} - w_2 {\rm{j}} + w_3 {\rm{k}} \cong a + \iGAj \mathbf{w} = a + w_1 \iGAj e_1 + w_2 \iGAj e_2 + w_3 \iGAj e_3 $.  Hamilton in fact originally conceived the quaternions as the quotient of two vectors, and indeed using Clifford algebra vectors we can explicate this idea, finding the quotient of two vectors
\be
\mathbf{v}/\mathbf{w} = \mathbf{v} \mathbf{w}/\mathbf{w}^2 = \frac{1}{\mathbf{w}^2 } \left (\textbf{v} \cdot \textbf{w}  + \textbf{v} \wedge \textbf{w} \right )
\ee
that lies on the even subalgebra and so isomorphic to the quaternions as asserted by Hamilton.

\subsubsection*{The square root}

We find that the same expression for square root of a multivector in two dimensions
\be \label{SquareRoot3D}
M^{\frac{1}{2}} = \frac{ M \pm |M| }{\sqrt{M + \cliffconj{M} \pm 2 |M|)}} ,
\ee
produces the square root in three dimensions. The full algebraic analysis of roots in three dimensions is quite extensive, however, as in complex number theory roots are more easily handled using the polar form of a number and we will find that the positive sign above will correspond to the principal value in the polar form $ \rme^{0.5 \log M} $, calculated using logarithms that are defined shortly.

\subsubsection*{Trigonometric form of a multivector}

\begin{defn}[Multivector argument]
We define the argument of a multivector $ \arg M = {\rm{arctan}} \left ( \frac{|F|}{Z} \right ) $, which is a multivalued function modulo $ 2 \pi $, but which also depends on which quadrant the point $ (Z,|F|) $ is in, where both $ Z $ and $ | F | $ are complex-like numbers. 
\end{defn}

Now, a multivector in $ C\ell \left (\Re^3 \right ) $ can be written in the form
\be \label{polarForm3DAssert}
M =  Z + F  = | M | \left ( \cos \phi + \hat{F} \sin \phi \right ) ,
\ee
where $ \phi = \arg M $, provided $ |M|, |F| \ne 0 $. We have defined $ \hat{F} = F/|F| $ that has the key property that $ \hat{F}^2 = -1 $.  This result can be confirmed by substituting $ \phi $, using the fact that $ \cos \phi = \frac{Z}{|M|} $ and $ \sin \phi = \frac{|F|}{|M|} $. Specifically, with $ F =  \mathbf{v} + \iGAj \mathbf{w}$ we find $  | F | = \sqrt{F \cliffconj{F}} = \sqrt{-\mathbf{v}^2 + \mathbf{w}^2 - 2 \iGAj \mathbf{v} \cdot \mathbf{w}}  $. The order of the factors is not as significant in three dimensions compared to two dimensions because the pseudoscalar $ j $ is commuting.

We will then find for integer powers $ p $ that
\be \label{polarFormPowers}
M^p = |M|^p \left ( \cos \phi + \hat{F} \sin \phi \right )^p =  |M|^p \left ( \cos p \phi + \hat{F} \sin p \phi \right ) ,
\ee
an extension of de~Moivre's theorem for multivectors in three dimensions.

\subsubsection*{Exponential map of a 3D multivector}

Now, given a three-dimensional multivector $ a + \mathbf{v} + \iGAj \mathbf{w} + \iGAj t = Z + F $, we find $ F^2 = (\mathbf{v} + \iGAj \mathbf{w})^2 = \mathbf{v}^2 - \mathbf{w}^2 + 2 \iGAj \mathbf{v} \cdot \mathbf{w} = - F \cliffconj{F} = -|F|^2 $. Now using Eq.~(\ref{exponentialMultivector}) and the fact that $ F = \hat{F} | F | $ we find
\bea \label{derivationExpMultivector}
\rme^{ M } = \rme^{ Z + F } & = & \rme^{Z} \rme^{ F } \\ \nonumber
& = & \rme^{Z} \left ( 1 + F + \frac{F^2 }{2!} + \frac{F^3}{3!} + \frac{F^4 }{4!} + \dots  \right ) \\ \nonumber
& = & \rme^{Z} \left ( 1 + \hat{F} | F | - \frac{| F |^2 }{2!}- \frac{ \hat{F} | F |^3}{3!} + \frac{| F |^4 }{4!} + \dots  \right ) \\ \nonumber
& = & \rme^{Z} \left ( \cos | F | + \hat{F}  \sin | F |   \right ) , \nonumber
\eea
and thus in a closed form.
If $ F^2 = 0 $, then referring to the second line of the derivation above, we see that all terms following $ F $ are zero, and so, in this case $ \rme^M = \rme^{ Z + F}  = \rme^{Z} (1+ F) $. The exponential function will also have the expected properties that $ \left (\rme^M \right )^{\cliffconj{}} = \rme^{\cliffconj{M}} $ and likewise for reversion and parity involutions.  A corollary of this result is that $ |\rme^M | = |\rme^Z \rme^F | = |\rme^Z | | \rme^F | = \rme^Z $.

We can thus write quite generally a multivector in polar form 
\be \label{polarForm3D}
M = a + \mathbf{v} + \iGAj \mathbf{w} + \iGAj t  =  | M | \rme^{ \phi \hat{F} }  ,
\ee
where $ \phi = \arg M $, provided $ |M|, |F| \ne 0 $, where clearly the exponent is multivalued. The polar form can also be expanded as $ | M | \rme^{ \phi \hat{F} }  =  | M | \left ( \cos \phi + \hat{F} \sin \phi \right ) $ and so equivalent to the trigonometric form shown in Eq.~(\ref{polarForm3DAssert}). We can therefore write a multivector $ M = \rme^{\log | M|} \rme^{\phi \hat{F} } = \rme^{\log | M| +  \phi \hat{F} } $ and defining the logarithm as the inverse of the exponential function, obtain the logarithm of a multivector 
\be \label{LogMulti}
\log M  = \log |M| + \phi \hat{F} ,
\ee
where $ \phi = \arg M $ and $ \hat{F} = \frac{ \mathbf{v} + \iGAj \mathbf{w}  }{|F|} $, provided $ |M|, |F| \ne 0 $.
Naturally, this will also coincide with the power series expansion of $ \log M = \log (1 + (M-1)) = (M-1) - \frac{1}{2} (M-1)^2 + \frac{1}{3} (M-1)^3 \dots $.
This leads to analogous results, as for complex numbers, that $ \log \iGAj M = \log M + \iGAj \pi/2 $ and $ \log ( -M) = \log M - \pi \hat{F} $. Some properties of the logarithm include $ \log (-1) = \iGAj \pi $ as well as the log of the trivector $ \log \iGAj = \frac{\pi}{2} \iGAj $, $ \log (b \iGAj) = \log b + \frac{\pi}{2} \iGAj $, and the log of a unit vector $ \log e_1 =  \iGAj ( 1 - e_1 ) \frac{\pi}{2}$ generalizing to $ \log \hat{\mathbf{v}} =  \iGAj ( 1 - \hat{\mathbf{v}} ) \frac{\pi}{2}$ and finally for a general vector $ \log \mathbf{v} = \log || \mathbf{v} || + \iGAj ( 1 - \hat{\mathbf{v}} ) \frac{\pi}{2}$.

The multivector logarithm is naturally a generalization of the well known result for quaternions, that can be recovered by setting $ \mathbf{v} = t = 0 $ giving  a multivector $ M = a + \iGAj \mathbf{w} $, with
\be \label{LogMultiQuatnerion}
\log(a + \iGAj \mathbf{w} ) = \log |q| + \phi \iGAj \hat{\mathbf{w}} = \log \sqrt{a^2 + \mathbf{w}^2 } +  {\rm{arctan}} \left ( \frac{\sqrt{\mathbf{w}^2}}{a } \right ) \frac{ \iGAj \mathbf{w}  }{\sqrt{\mathbf{w}^2}}  ,
\ee
where $ \phi =  {\rm{arctan}} \left ( \frac{\sqrt{\mathbf{w}^2}}{a } \right ) $, producing the quaternion logarithm as required. 
If we now set $ e_3 = 0 $ we find
\be \label{LogMultiComplex}
\log(a + w_3 e_{12} ) = \log \sqrt{a^2 + w_3^2 } +  {\rm{arctan}} \left ( \frac{w_3}{a } \right ) \frac{ w_3 e_{12} }{\sqrt{w_3^2} } = \log \sqrt{a^2 + w_3^2 } +  {\rm{arctan}} \left ( \frac{w_3}{a } \right ) e_{12} ,
\ee
the definition of the log of a complex number $ z = a + \iGAi w_3 $.

The nesting of real, complex numbers and quaternions within a multivector can be used to illustrate the Cayley-Dickson construction.  In the Cayley-Dickson construction, complex numbers are generated from pairs of real numbers, and subsequently quaternions are then generated from pairs of complex numbers, etc.

Now, the quaternions are the even subalgebra of $ C\ell \left ( \Re^3 \right ) $ and so we can write a quaternion
\be
q = a + \iGAj \mathbf{w} = \left ( a + w_3 e_{12} \right ) + \left ( w_2 + w_1 e_{12} \right ) e_{31} = z_1 + z_2 e_{31}
\ee
consisting now of a pair of complex-like numbers $ z_1 = a + w_3 e_{12} $ and $ z_2 = w_2 + w_1 e_{12} $.  We can then find the norm $ |q|^2 = \sqrt{|z_1|^2 + |z_2|^2 } $, and so derived from the norm of the constituent complex numbers.  Also, given two quaternions $ p = x_1 + x_2 e_{31} $ and $ q = y_1 + y_2 e_{31} $, where $ x_1, x_2, y_1, y_2 $ are complex like numbers in the form $ a + e_{12} b $, we find their product
\be
p q = (x_1 y_1 - x_2 \cliffconj{y}_2) + (x_1 y_2 + x_2 \cliffconj{y}_1 ) e_{31} .
\ee
This allows us to implement non-commutative quaternion multiplication using only commuting complex number arithmetic, which has advantages in numerical applications that utilize the already efficient implementation of complex number arithmetic.

Also, re-arranging the multivector the multivector in Eq.~(\ref{generalM3D}) to $ \left (a + \iGAj \mathbf{w} \right ) + \iGAj \left ( t - \iGAj \mathbf{v} \right ) = q_1 + \iGAj q_2 $, where $ q_1 = a + \iGAj \mathbf{w} $ and $ q_2 = t - \iGAj \mathbf{v} $ are quaternions, we have now written the multivector as a pair of quaternions. Though this is analogous to the Cayley-Dickson construction that will then produce the octonions, in our case we have formed rather the complexified quaternions, though both being eight dimensional spaces.
Hence in $ C\ell \left (\Re^3 \right ) $ we can identify the full multivector with the field of complexified quaternions, the even subalgebra $ a + \iGAj \mathbf{w} $ with the real quaternions, $ a + \iGAj t $ with the commuting complex numbers and the subalgebra $  a + v_1 e_1 + v_2 e_2 + w_3 e_1 e_2 $ with $ C\ell \left ( \Re^2 \right ) $.

The multivector logarithm highlights both the issue of multivaluedness and the non-commuting nature of multivectors.
Firstly, the non-commutativity implies that $ \rme^A \rme^B \ne \rme^{A+B} $ and hence $ \log A B \ne \log A + \log B $.  Also $ A^{n} B^{n} \ne (A B )^{n} $, unless $ A $ and $ B $ commute, where $ n $ an integer. 

Secondly, the issue of multivaluedness is typically addressed through defining the principle value of the logarithm and the use of Riemann surfaces, however with the multivector logarithm the multivaluedness can expand into two domains, of $ \hat{F} $ and $ \iGAj $.  This is because both $ \iGAj $ and $ \hat{F} $ square to minus one and commute with $ M $.
That is 
\be
M = \rme^{\log M } = \rme^{\log M + 2 n \pi \hat{F} } =  \rme^{\log M + 2 m \pi \iGAj } =  \rme^{\log M + \pi \hat{F} + \pi \iGAj } =  \rme^{\log M + (2 n + 1) \pi \hat{F} + (2 m + 1 )\pi \iGAj } 
\ee
where $ n $ and $ m $ are integers, where we can add even powers of $ \pi $.  Hence $ M = \rme^{\log M } $ whereas $ M \ne \log \left ( \rme^M \right ) $ due to the multivalued nature of the $ \log $ operation. 

Now, we can easily see that for $ n $ an integer that $ M^n = \left (\rme^{\log M} \right )^n = \rme^{\log M} \rme^{\log M} \dots \rme^{\log M} = \rme^{n \log M} $, which can be used as an alternative to Eq.~(\ref{polarFormPowers}).  
This leads us to define the multivector power
\be
M^P = \rme^{ \log ( M ) P } , 
\ee
where the power $ P $ is now generalized to a multivector.   This implies, for example, the power law $ \left ( M^P \right )^n = M^{ n P } $, where $ n \in  \mathbb{Z} $.
With this definition of power we can then define the log of a multivector $ Y $ to the multivector base $ M $ as
\be
\log_M Y =  \frac{1}{\log M}  \log Y .
\ee
Although, if the power $ P $ does not commute with $ \log M $ then we can also define a power as $ \rme^{ P \log ( M ) } $, that has a logarithm $ \log Y /\log M  $.  These expressions however need care due to the multivalued nature of the logarithm operation and the non-commutativity.

Now, using the logarithm function $  \rme^{ 0.5 \log M } $ 
\be
\rme^{ 0.5 (\log | M | + \phi \hat{F} ) } , \, \rme^{ 0.5 (\log | M | + \phi \hat{F} + \pi \hat{F} + \pi \iGAj ) } , 
\ee
we produce the two roots of $ M $ defined in Eq.~(\ref{SquareRoot3D}), as required.

\subsubsection*{Special cases}

We will now consider some special cases where we do have commuting multivectors, such as the case with two multivectors $ M $ and $ Z = a + \iGAj t $.
We then have that $ \log M^Z = Z \log M + \hat{F} m \pi + \iGAj n \pi $, where $ m,n \in  \mathbb{Z} $ add possible phase terms. We can eliminate the phase terms using the exponential function and write a more explicit expression as $ \rme^{\log M^z } = \rme^{z \log M } $.  We also then recover the well known relations that $ \rme^{Z } \rme^{ M } = \rme^{Z + M } $ and $ \log Z M = \log Z + \log M $.

A further special case\cite{Hestenes111} involves the product of two vectors $ \mathbf{a} $ and $ \mathbf{b} $, and we have from Eq.~(\ref{geometricProductRaw}) that
\be
\mathbf{a} \mathbf{b} = \rho \rme^{ \theta \hat{B} } = \rme^{ c + \theta \hat{B} }
\ee
where $ \theta $ is the angle between the two vectors, $ \cos \theta = \hat{\mathbf{a}} \cdot \hat{\mathbf{b}} $, $ c = \log \rho $ and $ \hat{B} $ is the unit bivector of the plane defined by the vectors.
We can then produce the result for two vectors $ \mathbf{a} $ and $ \mathbf{b}  $ that
\be
\log (\mathbf{a} \mathbf{b}) = \log ||\mathbf{a} || + \log ||\mathbf{b} || +  \theta \hat{B} = \frac{1}{2} \log \left ( \mathbf{a}^2 \mathbf{b}^2 \right ) + \theta  \frac{ \mathbf{a} \wedge \mathbf{b}}{|\mathbf{a} \wedge \mathbf{b}| } 
\ee
where $ \hat{B} = \frac{ \mathbf{a} \wedge \mathbf{b}}{|\mathbf{a} \wedge \mathbf{b}| }$ is the unit bivector formed by $ \mathbf{a} \wedge \mathbf{b} $ and $ \theta = \arcsin \frac{|\mathbf{a} \wedge \mathbf{b}|}{||\mathbf{a} || ||\mathbf{b} ||} $ is the angle between the two vectors.

\subsection*{Linear equations and linear functions}

We define a linear function over multivector variables
\be \label{LinearFnsDefn}
F(M) = \sum_{m = 1}^n R_m M S_m ,
\ee
where  $ R_m, R_m, M $ are multivectors. The series cannot in general be simplified due to non-commutativity. The case of $ n = 1 $, giving $ F(M) =  R M S  $ is particularly useful. For example, for the special case where $ R $ and $ S $ are vectors we have a reflection of a multivector
\be
M' = - \mathbf{v} M \mathbf{v} .
\ee
When $ R $ and $ S $ lie in the even subalgebra, isomorphic to the quaternions we have a rotation operation in three dimensions
\be
M' = R M \reversion{R} ,
\ee
where $ R \reversion{R} = 1 $.  The quaternions form a division algebra and so they are suitable to use as rotation operators that require an inverse.
There is also a generalization to describe rotations in $ \Re^4 $,
\be
M' = R M S ,
\ee
where $ M = x e_1 + y e_2 + z e_3 + \iGAj t $ represent a 4D Cartesian vector, with $ R \reversion{R} = S \reversion{S} = 1 $ .  

For the second case from Eq.~(\ref{LinearFnsDefn}) with $ n = 2 $ we have  $ F(M) = R M S + P M Q $. Now, premultiplying by $ S^{-1} $ from the right and $ P^{-1} $ from the left we produce $ Y = P^{-1} F(M) S^{-1} = P^{-1} R M + M Q S^{-1} $. Setting $ A = P^{-1} R $ and $ B = Q S^{-1} $ we produce
\be
Y = A M + M B,
\ee
which is called Sylvester's equation\cite{janovska2008linear} that can in general be solved for $ M $.  Assuming $ | A | \ne 0 $ (or alternatively $ | B | \ne 0 $) we firstly calculate $ A^{-1} Y \cliffconj{B} + Y = M \cliffconj{B} + A^{-1} M B \cliffconj{B} + A M + M B = M (B + \cliffconj{B}) + A^{-1} M B \cliffconj{B} + A M  $.  
Now $ B + \cliffconj{B} $ and $ B \cliffconj{B} $ are commuting complex-like numbers and so we can write $  \left ( \cliffconj{B} + B + A^{-1} B \cliffconj{B} + A \right ) M  = A^{-1} Y \cliffconj{B} + Y $, thus succeeding in isolating the unknown multivector $ M $.  
Hence we have the solution
\be
M = \left (B + \cliffconj{B} + A^{-1} B \cliffconj{B} + A  \right )^{-1}  \left ( A^{-1} Y \cliffconj{B} + Y \right ) .
\ee
This result is analogous to results using quaternions or matrices\cite{janovska2008linear}, though solved here for a general multivector, applicable in one, two or three dimensions. 

Regarding polynomial equations in multivectors, the fundamental theorem of algebra tells us that the number of solutions of a complex polynomial is equal to the order of the polynomial.  With multivector polynomials however, such as the simple quadratic equation $ M^2 + 1 = 0 $ we can find an infinite number of solutions.

A common operation in complex number theory is the process of `rationalizing the denominator' for a complex number $ \frac{1}{a + \iGAi b } $ that involves producing a single real valued denominator, given by $ \frac{a - \iGAi b}{a^2 + b^2} $. We can also duplicate this process for a multivector $ \frac{1}{a + \mathbf{v} + \iGAj \mathbf{w} + \iGAj t } $.  Now, $ 1/M = M^{-1} = \cliffconj{M}/(M \cliffconj{M}) $.  
Notice that $ M \cliffconj{M} $ is a complex-like number that we can now `rationalize' by multiplying the numerator and the denominator by $ \reversion{\left (M \cliffconj{M} \right )} $ forming $ \frac{1}{M} = \frac{\cliffconj{M} \reversion{\left (M \cliffconj{M} \right )}}{R} $, where $ R = M \cliffconj{M} \reversion{\left (M \cliffconj{M} \right )} $ is a scalar real value, as required. 

\subsubsection*{Trigonometric functions of multivectors in 3D}

The trigonometric functions in three-dimensions are more straightforward than in two-dimensions, because the unit imaginary $ j = e_{123} $ is commuting.  Using the general expressions in Eq.~(\ref{HyperTrigN}) and using $ M = Z + F $, we can once again write these expression in a closed form
\bea
\cosh M & = & \frac{1}{2} \left ( \rme^{ Z + F} +  \rme^{- Z -F} \right ) = \cos |F| \cosh Z  + \hat{F} \sin |F| \sinh Z  \\ \nonumber
\sinh M & = & \frac{1}{2} \left ( \rme^{ Z + F} -  \rme^{- Z -F} \right ) = \cos |F| \sinh Z + \hat{F} \sin |F| \cosh Z  .\nonumber
\eea
Using the commuting trivector $ \iGAj $, we can now write the trig relations given in the general case in Eq.~(\ref{TrigN}), as
\bea
\cos M & = & \frac{1}{2} \left ( \rme^{\iGAj M} +  \rme^{-\iGAj M} \right ) = \cosh |F| \cos Z  - \hat{F} \sinh |F| \sin Z   \\ \nonumber
\sin M & = & \frac{1}{2 \iGAj } \left ( \rme^{\iGAj M} -  \rme^{-\iGAj M} \right )  =  \cosh |F| \sin Z +  \hat{F} \sinh |F| \cos Z  . \nonumber
\eea
All the usual identities will hold such as $ \sinh \iGAj M = \iGAj \sin M $ and $ \cosh \iGAj M = \cos M $.  
Also, we can see that $ \sin M $ and $ \cos M $ are commuting, and so we can define without any difficulties with non-commutativity
\be
\tan M = \frac{\sin M}{\cos M } = \frac{ \tan Z + \hat{F} \tanh |F| }{1 - \hat{F} \tanh |F| \tan Z } ,
\ee
These expressions are easily calculated because both $ Z $ and $ |F | $ are complex-like numbers and so we can utilize the well known results from complex number theory.

We also have the results that for a general vector $ \mathbf{v} $ that $ \cos \mathbf{v} = \cos \sqrt{\mathbf{v}^2} = \cos ||\mathbf{v}|| $ that neatly generalizes scalar values to vector values.  That is, the $ \cos $ of a vector is the $ \cos $ of the length of the vector though $ \sin \mathbf{v} = \hat{\mathbf{v}} \sin || \mathbf{v} || $. For a field $ F = \mathbf{E} + \iGAj \mathbf{B} $ we have $ \cos F = \cos \sqrt{F^2} $.  

\subsubsection*{Inverse trigonometric functions}

Now, using the result that $ \rme^M = \cosh M + \sinh M $ substituting $ M = {\rm{arcsinh}} \, X $ then taking the $ \log $ of both sides we find
\be
{\rm{arcsinh}} \, X = \log \left (  \cosh \left ( {\rm{arcsinh}} \,X \right ) + X \right ).
\ee
Re-arranging $ \cosh^2 X - \sinh^2 X = 1 $ we find that $ \cosh \left ({\rm{arcsinh}} \, X \right ) = (1+X^2)^{1/2} $ and so 
\be
{\rm{arcsinh}} X = \log \left ( (1+X^2)^{1/2} + X \right ).
\ee
This will coincide with the power series $ {\rm{arcsinh}} \, X = \sum_{n=0}^{\infty} \frac{(-)^n (2n)!}{2^{2n} (n!)^2} \frac{X^{2n+1}}{2n+1} $. Similarly we find $ {\rm{arccosh}} X = \log \left ( X + (X^2 - 1)^{1/2} \right ) $ and $ {\rm{arctanh}} \, X  = \frac{1}{2} \log \left( (1 + X )/(1 - X) \right )  =  \frac{1}{2} \left ( \log \left(1 + X \right ) - \log \left (1 - X \right ) \right ) $ that coincides with the conventional power series $ {\rm{arctanh}} \, X = \sum_{n=0}^{\infty} \frac{X^{2n+1}}{2n+1} $.

Similarly, from the definitions of $ \cos $ and $ \sin $ we know that 
\be
\rme^{\iGAj M } = \cos M + \iGAj \sin M ,
\ee
and once again substituting  $ M = {\rm{arcsin}} \, X $ and using $ \cos \left ({\rm{arcsin}} \, X \right ) = (1-X^2)^{1/2} $ we find 
\be
{\rm{arcsin}} X = - \iGAj \log \left ( (1-X^2)^{1/2} + \iGAj X \right ).
\ee
Similarly we have $ {\rm{arccos}} X = - \iGAj \log \left ( X + \iGAj (1 - X^2 )^{1/2} \right ) $.

Hence we produce the result that $ {\rm{arcsinh}} \left ( \iGAj X \right ) = \iGAj {\rm{arcsin}} X $, $ {\rm{arccosh}} X = \iGAj {\rm{arccos}} X $ and $ {\rm{arctanh}} \left ( \iGAj X \right ) = \iGAj {\rm{arctan}} X $. Therefore, finally
\be
 {\rm{arctan}} \, X  = -\frac{\iGAj}{2} \log \left( (1 + \iGAj X )/(1 - \iGAj X) \right ) = -\frac{\iGAj}{2} \left ( \log \left( 1 + \iGAj X \right ) - \log \left (1 - \iGAj X \right ) \right ) .
\ee

As an example of solving multivector trigonometric equations, if we are asked to solve the equation $ \sinh M = 0 $ then we can proceed as follows.  Given $ \sinh M = \frac{1}{2} \left ( \rme^{M} - \rme^{- M} \right )  $ we therefore need to solve $ \rme^{M} = \rme^{- M} $ or $ \rme^{2 M} = 1 $.  Now
\be
 \rme^{2 M} = \rme^{2 \left ( a + \iGAj t \right ) + 2 \left ( \mathbf{v} + \iGAj \mathbf{w} \right )} = \rme^{2 a } \left ( \cos 2 t + \iGAj \sin 2 t \right ) \left ( \cos 2 |F| + \hat{F} \sin 2 |F| \right ) = 1 .
\ee
Hence we require $ a = 0 $, $ t = n \pi $ and $ | F | = m \pi $ so we have a solution $ M = m \pi \hat{F} + n \pi \iGAj $, where $ m,n \in \mathbb{Z}  $.  We can also identify a second solution in which both terms in the brackets are simultaneously negative $ M = \left (m + \frac{1}{2} \right ) \pi \hat{F} + \left ( n + \frac{1}{2} \right )\pi \iGAj $.

\subsubsection*{Inter-relationships in $ C \ell \left (\Re^3 \right ) $}

We have the well known result from complex number theory that $ \iGAi^{\iGAi} = \rme^{-\pi/2 } $ that is duplicated with the pseudoscalar in Clifford algebra, finding  $ \iGAi^{\iGAi} = \iGAj^{\iGAj} = \rme^{-\pi/2 } $. However with a more general multivector number now available we can also find other more general relationships.
For example, for a unit vector $ \hat{v} $ with $ \hat{v}^2 = 1  $, we find that $ \hat{v}^{\hat{v}} = \hat{v} $.  That is raising a unit vector to this unit vector power produces the same unit vector.  Alternatively, if we raise a unit vector to an orthogonal unit vector we find  $ \hat{v}^{\hat{v}^{\perp}} = 1 $.  

Also, consider the expression $ \left ( \cos \theta + \iGAj \sin \theta \right )^{\mathbf{v}} $, where $ \mathbf{v} = v_1 e_1 + v_2 e_2 + v_3 e_3 $ is a Cartesian vector, with $ \hat{v} = \mathbf{v}/\sqrt{\mathbf{v}^2} = \mathbf{v}/s $, then we find
\be
\left (\cos \theta +  \iGAj \sin \theta  \right )^{\mathbf{v}} = \rme^{ \mathbf{v} \log \left (\cos \theta + \iGAj \sin \theta  \right ) }= \rme^{  s \theta \iGAj  \hat{v}  } = \cos s \theta + \iGAj  \hat{v} \sin s \theta .
\ee
Now $ \hat{q} = \cos s \theta + \iGAj \hat{v} \sin s \theta  $ lies in the even sub-algebra and so is isomorphic to the quaternions with $ |q| = 1 $ and  $ \hat{z} = \cos \theta + \iGAj \sin \theta $ is isomorphic to the complex numbers, with $ \hat{z} $ representing a unit complex number.  We thus can write
\be \label{RosettaStone}
r \hat{z}^{\mathbf{v}} = q .
\ee
This formula thus links real numbers $ r \in \Re $, complex numbers $ z \in {\mathbb{C}} $, Cartesian vectors $ \mathbf{v} \in \Re^3 $ and quaternions $ q \in {\mathbb{H}} $ into a single relationship, a Rosetta stone for the algebra of three-dimensional space. 
Interpreting this formula, we can see that raising a unit complex number to a vector power produces a quaternion. A unit complex number being a rotation operator in the plane with a rotation of $ \theta $, when raised to a unit vector power in the direction $ \hat{v} $ produces a rotation operator rotating $ 2 \theta $ about the axis $ \hat{v} $.  Hence raising a complex number to a vector power $ \hat{v} $ converts a planar rotation operator into a three dimensional rotation operator about an axis $ \hat{v} $.
This relates to our previous discussion on the Cayley-Dickson construction that generates quaternions from complex numbers, but illustrates an alternate construction to achieve this.

These results are summarized in Table~~\ref{Rosetta}.

\begin{table*}[ht]
	\centering
\renewcommand{\arraystretch}{1.3}
\caption{Algebraic relations in three dimensions $ C\ell(\Re^3) $}  \label{Rosetta}

\begin{tabular}{|l|l|}
\hline
Main results  & Notes: $ \iGAj =e_{123}, \iGAi =e_{12} $, $ \mathbf{v} \in \Re^3 $  \\
\hline \hline
$ \iGAj^\iGAj = \iGAi^\iGAi = \rme^{-\pi/2} $ & Compare with $ \sqrt{-1}^{\sqrt{-1}} = \rme^{-\pi/2}  $ \\ 
$ (\iGAj \hat{v})^{\iGAj \hat{v}} = \rme^{-\pi/2} $ &  $  \hat{v}^2 = 1. $ E.g. $ (\iGAj e_1)^{\iGAj e_1} =  (e_2 e_3)^{e_2 e_3} = \rme^{-\pi/2} $ \\
$ (\iGAj \hat{v})^{\iGAj \hat{v}^{\perp}} = \iGAj \hat{w} $ & $ \hat{v}^{\perp} \cdot \hat{v} = 0 $, $ \hat{w} \hat{v}^{\perp} \hat{v} = \iGAj $. E.g. $ (\iGAj e_1)^{\iGAj e_3} = \iGAj e_2 $ \\ \hline
Powers of vectors &   $  \hat{v}^2 = 1 $  \\ \hline
$ \hat{v}^{\hat{v}} = \hat{v} $ & E.g. $ e_1^{e_1} = e_1 $ \\
$ \hat{v}^{\hat{v}^{\perp}} = 1 $ & E.g. $ e_1^{e_2} = e_2^{e_1} = 1 $  \\ 
$ {\hat{v}}^{ \iGAj \hat{v}^{\perp} } = 1 $ & E.g. $ e_2^{\iGAj e_3} = 1 $ \\ 
$ \mathbf{v}^{\frac{1}{2}}= \frac{1}{\sqrt{2 j \sqrt{\mathbf{v}^2}}} (\mathbf{v} + \iGAj \sqrt{\mathbf{v}^2} ) $ & E.g. $ e_1^{\frac{1}{2}}= \frac{1}{\sqrt{2 j }} (e_1 + \iGAj ) = \frac{1}{2} (1 - \iGAj ) (e_1 + \iGAj ) $\\ \hline
Trigonometric relationships &  \\ \hline
$ \cos \mathbf{v} = \cos || \mathbf{v} || $, $ \cos \hat{\mathbf{v}} = 1 $ & \\ 
$ {\rm{arcsinh}} \mathbf{v} = \log \left ( \mathbf{v} + \left (1+\mathbf{v}^2 \right )^{1/2} \right ) $ & \\ \hline
General relationships &  \\ \hline
$ r \hat{z}^{\mathbf{v}} = q $ & $ r \in \Re $, $ z \in {\mathbb{C}} $, $ |\hat{z}| = 1 $, $ q \in {\mathbb{H}} $ and $ \mathbf{v} \in \Re^3 $   \\ \hline
Special cases &  \\ \hline
$ \left ({\iGAj \hat{v}} \right )^{ \hat{v} } = \iGAj $ & E.g. $ ( \iGAj  e_3)^{e_3} =  ( e_1 e_2)^{e_3} =  \iGAj  $ \\ 
$ \iGAj^{\hat{v}} = j \hat{v} $ & E.g. $ \iGAj^{e_3} = \iGAj e_3 $ \\ \hline
\end{tabular}
\end{table*}

\subsection*{Multivector in one and four dimensions}

We can extend the sequence $ C\ell \left (\Re^3 \right ) $, $ C\ell \left (\Re^2 \right ) $ down to one dimension giving the multivector in $ C\ell \left (\Re^1 \right ) $
\be
M = a + v e_1
\ee
where $ a,v \in \Re $.  We now do not have a pseudoscalar, however most functions are still available.  It should also be noted that all multivectors are now commuting because we only have a single algebraic variable $ e_1 $.
The list of functions can be deduced from the expressions for two or three dimensions through setting $ e_2 = e_3 = 0 $, that then gives $ |F| = | v e_1 | = \iGAj v $ and $ \hat{F} = - \iGAj e_1 $. 

The four dimensional case $ C\ell \left (\Re^4 \right ) $ is significantly harder than three dimensions, due to a larger sixteen dimensional space as well as a non-commuting pseudoscalar $ I = e_{1234} $.  We have a multivector $ M = a + \mathbf{v} + \mathbf{B} + I \mathbf{w} + I t $, where $ \mathbf{v},\mathbf{w} \in  \Re^4  $ and the bivectors $ \mathbf{B} = \sum_{i,j=1}^4 b_{ij} e_i e_j $. We have Clifford conjugation $ \cliffconj{M} = a - \mathbf{v} - \mathbf{B} + I \mathbf{w} + I t $ as well as a new involution $ M^{\sharp} = a - \mathbf{v} + \mathbf{B} - I \mathbf{w} - I t $.  We can then find a multivector amplitude $ | M | = \left (M \cliffconj{M} \left (M \cliffconj{M} \right )^{\sharp} \right )^{1/4}  $ that allows us to find an inverse $ M^{-1} = \cliffconj{M} (M \cliffconj{M})^{\sharp} /|M|^4 $ provided  $ | M | \ne 0 $. 

If we seek the next space that has a commuting pseudoscalar that squares to minus one we need to go to a seven dimensional space.  This space consists of eight grades with a total of $ 2^7 = 128 $ elements.
The difficulties with spaces other than dimension of three, serves to illustrate the elegant mathematical framework that $ C \ell \left (\Re^3 \right ) $ provides.

\newpage

\section*{Results}

\begin{table*}[ht]
	\centering
\renewcommand{\arraystretch}{1.3}
\caption{Multivector functions in two dimensions $ M \in C\ell \left (\Re^2 \right ) $}  \label{twoDFunctions}
\begin{tabular}{|l|l|}
\hline
Main results  & Notes ($ \iGAi = e_{12} $ non-commuting)   \\
\hline \hline
$ M = a + \mathbf{v} + \iGAi b = a + F $ & Define $ \mathbf{v} \in  \Re^2  $, $ a,b \in  \Re $, $ F = \mathbf{v} + \iGAi b $  \\
$ \cliffconj{M} = a - \mathbf{v} - \iGAi b = a - F $ & Conjugation  \\
$ | M | = \sqrt{M \cliffconj{M}} = \sqrt{a^2 - \mathbf{v}^2 + b^2  } $ & Amplitude    \\
$ M^{-1} = \cliffconj{M} /(M \cliffconj{M}) $ : $ | M | \ne 0 $ & Inverse \\ 
$ \phi = \arg M = {\rm{arctan}} \left ( \frac{|F|}{a} \right ) $ &  Argument, $ | F | = \sqrt{b^2 - \mathbf{v}^2  } $ \\ 
$ M = \left ( \cos \phi +  \hat{F} \sin \phi \right ) |M| $ : $ | M |, |F| \ne 0 $  &  Trigonometric form, $ \hat{F} = F/|F| $ \\ 
$ M^p =  \left ( \cos p \phi + \hat{F} \sin p \phi \right ) |M|^p $ &  Integer powers,  $ p \in {\mathbb{N}}  $ \\ 
$ \rme^M  = \rme^{a} \left ( \cos | F | + \hat{F}  \sin | F |   \right )  $ & Exponential, $\rme^M  \equiv \Sigma_{n=0}^{\infty}  \frac{M^{n}}{n!} $ \\ 
   & If $ |F| = 0 $ then $ \rme^M = \rme^a (1+ F) $\\ \hline
$ M^{\frac{1}{2}} =  \frac{ M \pm |M| }{\sqrt{M + \cliffconj{M} \pm 2 |M|}} $ : $ | M | \in \Re $ & Square root \\ 
$ M =  | M | \rme^{\hat{F} \phi  } $ : $ | M | \in \Re \ne 0 $ & Polar form, $ \hat{F}^2 = -1 $ \\
$ M^{x} = |M|^{x} \left ( \cos  x \phi +\hat{F} \sin x \phi \right ) $ : $ | M | \in \Re $ &  Real powers,  $ x \in \Re $    \\
$ \log_e M = \log_e |M|  +  \hat{F} \phi $  & Logarithm  \\ 
$ M^P = \rme^{ \log ( M )  P } $ or $ \rme^{ P \log ( M ) } $ &  General powers \\ \hline
Hyperbolic trigonometric functions  &   $ M = a + F =  a+ \hat{F} |F| $  \\ \hline
 $ \cosh M  = \cos |F| \cosh a  + \hat{F} \sin |F| \sinh a  $ &  $ \rme^M = \cosh M + \sinh M $  \\ 
 $ \sinh M  = \cos |F| \sinh a + \hat{F} \sin |F| \cosh a $ &  $ \cosh^2 M - \sinh^2 M = 1 $ \\ \hline
 $ {\rm{arcsinh}} \, M  = \log \left (M + \left (1 + M^2 \right )^{\frac{1}{2}} \right ) $ & Inverse hyperbolic sin \\
 $ {\rm{arccosh}} M = \log \left ( M + (M^2 - 1)^{1/2} \right ) $ &  Inverse hyperbolic cos  \\ \hline
\end{tabular}
\end{table*}

\newpage

\begin{table*}[ht]
	\centering
\renewcommand{\arraystretch}{1.3}
\caption{Multivector functions in three dimensions $ C\ell \left (\Re^3 \right ) $}  \label{threeDFunctions}
\begin{tabular}{|l|l|}
\hline
Main results  & Notes ($ \iGAj =e_{123} $ commuting) \\
\hline \hline
$ M = a + \mathbf{v} + \iGAj \mathbf{w} + \iGAj t $ & $ F = \mathbf{v} + \iGAj \mathbf{w} $, $ Z = a + \iGAj t $, $ \mathbf{v},\mathbf{w} \in  \Re^3  $ \\
$ \cliffconj{M} = a - \mathbf{v} - \iGAj \mathbf{w} + \iGAj t $ & Conjugation \\
$ | M | = \sqrt{M \cliffconj{M}} = \sqrt{a^2 - \mathbf{v}^2 + \mathbf{w}^2 - t^2 + 2 \iGAj (a t - \mathbf{v} \cdot \mathbf{w} ) } $ &  Amplitude \\
$ M^{-1} = \cliffconj{M} /(M \cliffconj{M}) $ : $ | M | \ne 0 $ &  Inverse \\ \hline
$ \phi = \arg M = {\rm{arctan}} \left ( \frac{|F|}{Z} \right ) $ & Argument \\
$ M =  | M | \rme^{ \phi \hat{F} } =  | M | \left ( \cos \phi + \hat{F} \sin \phi \right )  $ : $ | M |,|F| \ne 0 $ & Polar form, $ \hat{F} = F/|F| $, $ \hat{F}^2 = -1 $ \\
$ M^{\frac{1}{2}} =  \frac{ M \pm |M| }{\sqrt{M + \cliffconj{M} \pm 2 |M|}} $ & Square root  \\ 
$ M^{x} = |M|^{x} \left ( \cos  x \phi +\hat{F} \sin x \phi \right ) $ & Complex powers $ x \in \Re \oplus \bigwedge^3 \Re^3 $    \\
$ \rme^M = \rme^{Z + F} = \rme^{Z} \left ( \cos | F | + \hat{F}  \sin | F |   \right ) $ & Exponential \\
 & If $ |F| = 0 $ then $ \rme^M = \rme^Z (1+ F) $\\
$ \log_e M = \log_e |M| + \phi \hat{F}  $ & Logarithm \\
$ M^P = \rme^{ \log ( M )  P } $ or $ \rme^{ P \log ( M ) } $ &  General powers  \\ \hline
Hyperbolic/Trigonometric functions  &   $ M = Z + F =  Z+ \hat{F} |F| $  \\ \hline
 $ \cosh M = \cos |F| \cosh Z  + \hat{F} \sin |F| \sinh Z  $ &   $ \rme^{ M }= \cosh M + \sinh M $  \\
 $ \sinh M = \cos |F| \sinh Z + \hat{F} \sin |F| \cosh Z $ &  $ \cosh^2 M - \sinh^2 M = 1 $ \\
 $ \cos M =  \cosh \iGAj M = \cosh |F| \cos Z - \hat{F} \sinh |F| \sin Z  $ & $ \rme^{\iGAj M }= \cos M + \iGAj \sin M $ \\
 $ \sin M = - \iGAj \sinh \iGAj M = \cosh |F| \sin Z +  \hat{F} \sinh |F| \cos Z $ &   $ \cos^2 M + \sin^2 M = 1 $ \\ \hline
 $ {\rm{arcsinh}} \, M  = \log \left (M + \left (1 + M^2 \right )^{\frac{1}{2}} \right ) $ &  $ {\rm{arcsinh}} \left ( \iGAj M \right ) = \iGAj {\rm{arcsin}} M $ \\
 $ {\rm{arccosh}} M = \log \left ( M + (M^2 - 1)^{1/2} \right ) $ &  $ {\rm{arccosh}} X = \iGAj {\rm{arccos}} X $ \\
 $ {\rm{arctanh}} \, M  = \frac{1}{2} \log \left( (1 + M )/(1 - M) \right ) $ &  $ {\rm{arctanh}} \left ( \iGAj M \right ) = \iGAj {\rm{arctan}} M  $ \\
\hline 
\end{tabular}
\end{table*}

\newpage


\section*{Discussion}

In this paper, we explore the elementary functions when generalized to act over the space of Clifford multivectors in two and three dimensions, refer Table~\ref{twoDFunctions} and Table~\ref{threeDFunctions}.  Two key points that need to be kept in mind when working with multivectors, is firstly their non-commutativity and secondly multivaluedness, as found with the  $ \log $ function and the square root functions, for example.  According to Frobenius' theorem the only associative division algebras are the reals, complex numbers and quaternions and conveniently these form subalgebras within $ C\ell \left (\Re^3 \right ) $. For the multivector, as represented in Eq.~(\ref{generalM3D}), the reals form the scalar component $ a $, the complex-like numbers represented by the scalar and trivector components $ a + \iGAj t $ and the quaternions by the even subalgebra $ a + \iGAj \mathbf{w} $. The full multivector in three dimensions, on the other hand, do not form a division algebra as the inverse operation is not defined for the multivectors with zero amplitude. 

We have used the symbols $ \iGAi $ and $ \iGAj $ to replace the unit imaginary in two and three dimensions respectively. The use of the bivector and trivector for this purpose allows us to duplicate imaginary quantities with real algebraic entities thus remaining within a real space. Also the quantities $ \iGAi = e_1 e_2  $ and $ \iGAj = e_1 e_2 e_3 $ can be endowed with specific geometrical meaning as a unit area and a unit volume respectively.

We find that in two and four dimensions  with a non-commuting pseudoscalar the elementary functions lack generality, whereas in three dimensions with a commuting pseudoscalar the functions are defined more generally.  For example, in three dimensions the exponential form (polar decomposition) exists for all multivectors provided $ |M|,|F| \ne 0 $, whereas in two dimensions, they only exist if $ |M| $ is real. In fact, in three dimensions with a commuting pseudoscalar allows us to identify $ C\ell \left (\Re^3 \right ) $ with a complex algebra, isomorphic to complexified quaternions.

We find that because the complex numbers and quaternions appear as subalgebras within the more general multivector, we can explore their mutual inter-relationships within this context.  We find that we can link the three algebraic systems of vectors, quaternions and complex numbers into a single expression, finding that a complex number raised to a vector power produces a quaternion, as shown in Eq.~(\ref{RosettaStone}).  A relationship between complex numbers and quaternions is already provided by the Cayley-Dickson construction of quaternions from complex numbers, however our expression is more explicit generating quaternions from raising a complex number to the power of a Cartesian vector.  Inspecting the list of functions we also identify a single formula that produces the square root, amplitude and inverse in two and three dimensions. Also vectors are given a much more versatile formulation in Clifford algebra compared to Gibbs formulation of vectors, and indeed we can explore various vector expressions, such as raising a vector to a vector power, as well as logarithms and trigonometric relationships with vectors, these and other relationships listed in Table~~\ref{Rosetta}.
 We also find that the elementary functions can be defined using a single involution of Clifford conjugation, although for convenience two other involutions of reversion and space inversion are also defined.  Naturally, the elementary functions over complex numbers and quaternions can be recovered as special cases from the three dimensional case.

The two dimensional algebra has the even subalgebra isomorphic to the complex numbers and has application within planar geometry, such as planar waveguides, and the three-dimensional algebra has the obvious application to three-dimensional space and forms an elegant space due to the commuting pseudoscalar, as well as possessing complex numbers and quaternions as subalgebras that can be utilized for rotations and containing Cartesian vectors that perform reflections.  

The multivector can also provide a unifying perspective on physical phenomenon in three dimensions.
For example, the Schr{\"o}dinger wave equation defines a wave function over the complex field that was extended by Pauli to include spin, generalizing the wave function to a quaternion.  Dirac completed this process producing the full relativistic wave equation for spin, that generalized the quaternionic wave function of Pauli to an eight-dimensional wave function, isomorphic to the eight-dimensional multivector\cite{Boudet} in $ C\ell \left ( \Re^3 \right ) $, as shown in Eq.~(\ref{generalM3D}). All three of these wave functions are subalgebras within the multivector in three dimensions, thus providing a unified picture. This example serves to illustrate the value of detailing the functions over multivector variables as undertaken in this paper.


\section*{Acknowledgments}


%
%
%



\section*{Figure Legends}
%

\section*{Tables}
%
%
%

\section*{Supporting Information Legends}
%
%

\end{document}